\documentclass[11pt, a4paper]{article}
\usepackage{amsmath, amsthm, amssymb, url}
\usepackage[margin=2.5cm]{geometry}
\usepackage{multirow}
\usepackage{tikz, mathrsfs, enumerate}
\usepackage{authblk}
\usepackage{slashbox}

\newcommand{\Sym}{\mathrm{Sym}}

\newcommand{\Fix}{\mathrm{Fix}}

\newcommand{\Aut}{\mathrm{Aut}}

\theoremstyle{plain}

\newtheorem{corollary}{Corollary}
\newtheorem{lemma}{Lemma}

\newtheorem{theorem}{Theorem}

\theoremstyle{definition}

\newtheorem{remark}{Remark}

\begin{document}

\title{The number of configurations in the full shift with a given least period}
\author{Alonso Castillo-Ramirez\footnote{Email: alonso.castillor@academicos.udg.mx} \ and Miguel S\'anchez-\'Alvarez\footnote{Email: miguel.sanchez1273@academicos.udg.mx }  \\
\small{Department of Mathematics, University Centre of Exact Sciences and Engineering,\\ University of Guadalajara.} }

\maketitle

\begin{abstract}
For any group $G$ and any set $A$, consider the shift action of $G$ on the full shift $A^G$. A configuration $x \in A^G$ has \emph{least period} $H \leq G$ if the stabiliser of $x$ is precisely $H$. Among other things, the number of such configurations is interesting as it provides an upper bound for the size of the corresponding $\Aut(A^G)$-orbit. In this paper we show that if $G$ is finitely generated and $H$ is of finite index, then the number of configurations in $A^G$ with least period $H$ may be computed by using the M\"obius function of the lattice of subgroups of finite index in $G$. Moreover, when $H$ is a normal subgroup, we classify all situations such that the number of $G$-orbits with least period $H$ is at most $10$.    \\


\end{abstract}

%\section{Declarations}
%
%\textbf{Funding}: The first author of this paper was supported by a CONACYT Basic Science Grant (No. A1-S-8013) from the Government of Mexico. The second author of this paper was supported by a CONACYT National Scholarship for PhD.  \\
%
%\textbf{Conflicts of interest/Competing interests}: Not applicable. \\
%
%\textbf{Availability of data and material}: Not applicable.  \\
%
%\textbf{Code availability}: Not applicable. 

\section{Introduction}\label{intro}

Let $G$ be a group and let $A$ be a set. Consider the set $A^G$ of all functions from $G$ to $A$ equipped with the \emph{shift action} of $G$, defined by
\[ (g \cdot x) (h) := x(g^{-1}h),  \]
for all $g, h \in G$ and $x \in A^G$. Although we shall not focus on this, the set $A^G$ is usually seen as a topological space with the product topology of the discrete topology on $A$. 

The $G$-space $A^G$ is a fundamental object in areas such as symbolic dynamics and the theory of cellular automata (e.g. see \cite{CSC10,LM95}). Following \cite{CSC10}, we call the elements of $A^G$ \emph{configurations}. For any $x \in A^G$, the \emph{stabiliser} $G_x$ of $x$ and the \emph{$G$-orbit} $Gx$ of $x$ are defined as follows:
\[ G_x := \{ g \in G : g \cdot x = x \} \quad \text{ and } \quad Gx := \{ g \cdot x \in A^G  : g \in G \}. \] 

For a subgroup $H$ of $G$, a configuration $x \in A^G$ has \emph{period} $H$, or is \emph{$H$-periodic}, if $h \cdot x = x$ for all $h \in H$, or, equivalently, if $H \leq G_x$. Denote by $\Fix(H)$ the subset of $A^G$ consisting of all $H$-periodic configurations. It is known (see \cite[Proposition 1.3.3]{CSC10}) that $\Fix(H)$ is in bijection with $A^{H\setminus G}$, where $H\setminus G = \{ Hg :g \in G \}$ is the set of rights cosets of $H$ in $G$. Hence, it follows that $\vert \Fix(H) \vert = \vert A \vert ^{[G:H]}$, where $[G:H] := \vert H\setminus G \vert $ is the index of $H$ in $G$. In particular, the configurations whose period is the trivial subgroup of $G$ are known as \emph{aperiodic points}, and have been used in \cite{GJS16} as powerful tools to study the dynamics in $A^G$ and its \emph{subshifts}, or \emph{subflows} (i.e. closed $G$-equivariant subsets of $A^G$).

We say that $x \in A^G$ has \emph{least period}, or \emph{fundamental period}, $H$ if $G_x = H$ (c.f. \cite[Definition 1.1.3.]{LM95}). In this paper we are interested in the number $\psi_H(G;A)$ of configurations with least period $H$: 
\[ \psi_H(G;A) := \vert \{ x \in A^G : G_x = H \} \vert. \]  
If $x,y \in A^G$ satisfy that $y=g\cdot x$, then $G_y = g G_x g^{-1}$; hence, it is sometimes convenient to consider the $G$-invariant set $\{ x \in A^G : [G_x] = [H] \}$, where $[H] := \{gHg^{-1} : g \in G \}$ is the conjugacy class of $H$, and its cardinality
\[ \psi_{[H]}(G;A) := \vert \{ x \in A^G : [G_x] = [H] \} \vert. \]
As $\psi_{H}(G;A) = \psi_{gHg^{-1}}(G;A)$ for all $g \in G$, we have
\[ \psi_{[H]}(G;A) = \vert [H] \vert \;  \psi_H(G;A).   \]
Finally, we also consider the number of $G$-orbits whose stabiliser is conjugate to $H$:
\[ \alpha_{[H]}(G;A) := \vert \{ Gx : [G_x] = [H] \} \vert. \]
By the Orbit-Stabiliser Theorem (\cite[Theorem 7.2.1]{R12}), all $G$-orbits inside $\{  x \in A^G : [G_x] = [H]\}$ have size $[G:H]$; therefore, we have
\[ \alpha_{[H]}(G;A) \; [G:H] = \psi_{[H]}(G;A).   \]  
  
Besides being interesting for their own right, the above numbers have connections with the structure of the automorphism group of $A^G$. Recall that a map $\tau : A^G \to A^G$ is \emph{$G$-equivariant} if $\tau(g \cdot x) = g \cdot \tau(x)$, for all $g \in G$, $x \in A^G$. Let $\Aut(A^G)$ the group of all $G$-equivariant homeomorphisms of $A^G$. By the Curtis-Heldund Theorem (\cite[Theorem 1.8.1]{CSC10}), $\Aut(A^G)$ is the same as the group of invertible cellular automata of $A^G$. It follows by $G$-equivariance that for every $\tau \in \Aut(A^G)$, $x \in A^G$, we have $G_{x} = G_{\tau(x)}$. Thus, $\psi_{G_x}(G;A)$ is an upper bound for the cardinality of the $\Aut(A^G)$-orbit of $x$. Moreover, if the group $G$ is finite, the structure of $\Aut(A^G)$ was described in \cite[Theorem 3]{CRG19} as
\begin{equation}\label{aut-desc}
\Aut(A^G) \cong \prod_{i=1}^{r} ( (N_G(H_i)/H_i) \wr \Sym_{\alpha_i}),
\end{equation}
where $[H_1], \dots, [H_r]$ is the list of all different conjugacy classes of subgroups of $G$, and $\alpha_i = \alpha_{[H_i]}(G;A)$, as defined above. Hence, the structure of $\Aut(A^G)$ completely depends on the quotient groups $N_G(H_i)/H_i$, which may be easily calculated by knowing the group $G$, and the integers $\alpha_{[H_i]}(G;A)$, which depend on $\psi_H(G;A)$. Finally, in \cite{BK87,BLR88}, the sets of points of a given least period were a fundamental tool in the study of automorphism groups of shifts of finite type, which include the group $\Aut(A^{\mathbb{Z}})$.

As $\psi_H(G;A)$ is finite if and only if $[G:H]$ is finite (see Lemma \ref{le-finite-index} below), we shall focus on finite index subgroups of $G$. In the first part of this paper, we prove that, when $G$ is finitely generated, the poset $L(G)$ of finite index subgroups of $G$ is a locally finite lattice, so we use M\"obius inversion to show that 
\begin{equation}\label{eq-main}
  \psi_H(G;A) =  \sum_{H \leq K \leq G} \mu(H,K) \vert A \vert^{[G:K]},  
\end{equation}
where $\mu$ is the M\"obius function of $L(G)$. In the second part of this paper, we note that if $H$ is a normal subgroup, then $\psi_H(G;A) = \psi_1(G/H;A)$ and $\alpha_{[H]}(G;A) = \alpha_{[1]}(G/H;A)$. Hence, by computing the M\"obius function of the subgroup lattice of all finite groups of size up to $7$, we classify under which situations we have $\alpha_{[H]}(G;A) \leq 10$.  

Our work generalises previous results known in the literature. When $G = \mathbb{Z}_n$ is a cyclic group and $H=1$ is the trivial subgroup, $\alpha_{[1]}(\mathbb{Z}_n;A)$ is equivalent to the number of aperiodic necklaces of length $n$, and equation (\ref{eq-main}) gives the so-called Moreau's necklace-counting function \cite{M72}. Moreover, $\alpha_{[1]}(\mathbb{Z}_n;A)$ is also equivalent to the number of Lyndon words of length $n$ (see Sec. 5.1. in \cite{L97}). For a finite group $G$, this equation may be derived using the result of Sec. 4 in \cite{Ke99}. However, as far as we know, equation (\ref{eq-main}) had not been derived when $G$ is an arbitrary finitely generated group.

%%%%%%%%%%%%%%%%%%%%%%%%%%%%%%%%%%%%
\section{Periodic configurations when $G$ is finitely generated}

For the rest of the paper, let $A$ be a set with at least two elements and assume that $\{0,1 \} \subseteq A$. We begin by justifying our claim that $\psi_H(G;A)$ is finite if and only if $[G:H]$ is finite. 

\begin{lemma}\label{le-finite-index}
Let $G$ be a group and let $H$ be a subgroup of $G$. Then $\psi_H(G;A)$ is finite if and only if $[G:H]$ is finite.
\end{lemma}
\begin{proof}
If $[G:H]$ is finite, then $\psi_H(G;A)$ is clearly finite, as every configuration with least period $H$ is contained in $\Fix(H)$ and $\vert \Fix(H) \vert = \vert A \vert^{[G:H]} < \infty$. 

Conversely, suppose that $[G:H]$ is infinite. Let $T \subseteq G$ be a transversal for the set of right cosets of $H$ in $G$, i.e., $T$ contains exactly one element from each right coset of $H$ in $G$. It is clear that $\vert T \vert = [G:H]$. For each $s \in T$, consider the configuration $x_s \in A^G$ defined by
\[ x_s(g) = \begin{cases}
1 & \text{ if } g \in Hs \\
0 & \text{ otherwise}
\end{cases} ,\]
for any $g \in G$. Given $h \in H$, then $h \cdot x_s(g) = x_s(h^{-1}g) = x_s(g)$, as $h^{-1}g \in Hs$ if and only if $g \in Hs$. Hence, $H \leq G_{x_s}$. On the other hand, if $k \in G_{x_s}$, then $k \cdot x_s = x_s$; in particular we have $(k \cdot x_s)(s) =x_s(k^{-1}s) = x_s(s) = 1$, which implies that $k^{-1}s \in Hs$. Therefore, $k \in H$, which shows that $G_{x_s}=H$. As $\vert T \vert = [G:H]$ is infinite, we have constructed infinitely many different configurations with least period $H$, which establishes that $\psi_{H}(G;A)$ is infinite.   
\end{proof}

We shall recall some basic definitions on posets; for further details see \cite[Ch. 3]{S12}. Recall that a \emph{partially ordered set}, or a \emph{poset}, is a set $P$ equipped with a partial order relation $\leq$. Given $s, t \in P$ with $s \leq t$, define the closed interval $[s,t] := \{ u \in P : s \leq u \leq t \}$. We say that $P$ is \emph{locally finite} if every closed interval of $P$ is finite. A \emph{chain} of $P$ is a subposet $S$ of $P$ that is totally ordered, i.e. any two elements of $S$ are comparable. For $t \in P$, the \emph{principal order ideal generated by $t$} is $\Lambda_t := \{ s \in P : s \leq t \}$, and the \emph{principal dual order ideal generated by $t$} is $V_t := \{s \in P : s \geq t \}$.

A \emph{lattice} is a poset $L$ for which every pair of elements $s,t \in L$ has a lest upper bound, denoted by $s \lor t$ and read $s$ \emph{join} $t$, and a greatest lower bound, denoted by $s \land t$ and read $s$ \emph{meet} $t$.  

The \emph{M\"obius function} of a locally finite poset $P$ is the map $\mu : P \times P \to \mathbb{Z}$ defined inductively by the following equations:
\begin{align*}
\mu(a,a) & = 1, \ \ \forall a \in P, \\
\mu(a,b) &= 0, \ \ \forall a \not\leq b, \\
\sum_{a \leq c \leq b} \mu(a,c) &= 0, \ \ \forall a < b.
\end{align*}

The M\"obius function is the inverse of the zeta function of a locally finite poset, and it importantly satisfies the so-called M\"obius inversion formula (see \cite[Sec. 3.7]{S12}). In this section we shall use the dual form of the M\"obius inversion formula \cite[Proposition 3.7.2]{S12}. 

\begin{theorem}[M\"obius inversion formula, dual form]\label{Mobius}
Let $P$ be a poset for which every principal dual order ideal $V_t$ is finite. Consider functions $f,g : P \to K$, where $K$ is a field. Then
\[ g(t) = \sum_{s \geq t} f(s), \quad \forall t \in P, \]
if and only if
\[ f(t) = \sum_{s \geq t} g(s) \mu(t,s), \quad \forall t \in P.  \]
\end{theorem}

For any group $G$, it is standard to consider the poset of all subgroups of $G$ ordered by inclusion. Here, we shall consider the poset $L(G)$ of all subgroups of $G$ of finite index ordered by inclusion. The following is a key observation for this section.

\begin{lemma}\label{le-key}
The poset $L(G)$ is a lattice. Furthermore, if $G$ is finitely generated, then for every $H \in L(G)$, the principal dual order ideal $V_H = \{ K \leq G : H \leq K \}$ is finite, so $L(G)$ is a locally finite lattice. 
\end{lemma} 
\begin{proof}
We shall show that $L(G)$ is a sublattice of the subgroup lattice of $G$ by showing that it is closed under the join, given by $H \lor J = \langle H \cup J \rangle$, and the meet, given by $H \land J = H \cap J$. 

Let $H$ and $K$ be subgroups of $G$ such that $H \leq K$. It is well-known (see, for instance \cite[Theorem 3.1.3]{R12}) that the indices of $H$ and $K$ in $G$ satisfy, as cardinal numbers, that
\[ [G:H] = [G:K] [K:H]. \]
Hence, if $[G:H]$ is finite, then $[G:K]$ must be finite. This implies that for any $H, J \in L(G)$, then $\langle H \cup J \rangle \in L(G)$. On the other hand, it is also well-known (see, for instance \cite[Theorem 3.1.6]{R12}) that the intersection of subgroups of finite index has finite index, so $H \cap J \in L(G)$, and the first part of the lemma follows. 

For the second part, for any $H \in L(G)$ and $K \in V_H$, the index of $K$ in $G$ must be a divisor of $[G:H]$. The result follows as in a finitely generated group there are only finitely many subgroups of a given finite index (this is a well-known theorem by M. Hall \cite{H50}; see also \cite[Theorem 4.20]{R12}).  
\end{proof}

The previous lemma allows us to use the M\"obius inversion formula for the poset $L(G)$ when $G$ is finitely generated. Let $\mu$ be the M\"obius function of $L(G)$. 

\begin{theorem}\label{th-main}
Let $G$ be a finitely generated group, let $H$ be a subgroup of $G$ of finite index, and let $A$ be a finite set. Then,
\[  \psi_H(G;A) =  \sum_{H \leq K \leq G} \mu(H,K) \vert A \vert^{[G:K]}.  \]
\end{theorem}
\begin{proof}
It follows from the definitions that
\[ \vert \Fix(H) \vert = \sum_{K \geq H} \psi_K(G;A) .  \]
By Lemma \ref{le-key} this summation is finite and we may use Theorem \ref{Mobius}, with $g(H) = \vert \Fix(H) \vert$ and $f(K) =  \psi_K(G;A)$. Therefore, we obtain
\[ \psi_H(G;A) = \sum_{K \geq H} \mu(H,K) \vert \Fix(K) \vert .  \]
The result follows as $\vert \Fix(K) \vert= \vert A \vert^{[G:K]}$ by \cite[Proposition 1.3.3]{CSC10}.
\end{proof}

\begin{remark}
Note that, for any $H,J \in L(G)$, the value of $\mu(H,J)$ only depends on the on the interval $[H,J]$. Hence, $\psi_H(G;A)$ may be calculated by only knowing the subposet $[H,G]$.  
\end{remark}

\begin{corollary}
With the notation of Theorem \ref{th-main}, suppose that the interval from $H$ to $G$ consists of a chain $H = H_0 < H_1 < \dots < H_k=G$. Then,
\[ \psi_H(G;A) = \vert A \vert^{[G:H]} - \vert A \vert^{[G:H_1]}.   \]
In particular, if $H$ is a maximal subgroup of $G$, then
\[ \psi_H(G;A)  = \vert A \vert^{[G:H]} - \vert A \vert.   \]
\end{corollary}
\begin{proof}
By Theorem \ref{th-main},
\[  \psi_H(G;A) =  \sum_{i=0}^k \mu(H,H_i) \vert A \vert^{[G:H_i]}. \]
Now, by the definition of the M\"obius function, 
\begin{align*}
\mu(H,H_0) & = 1, \\
\mu(H,H_1) & = -1, \\
\mu(H,H_i) & = 0, \quad  \quad \forall i=2,3,\dots, k. 
\end{align*}
The result follows. 
\end{proof}

\begin{corollary}
With the notation of Theorem \ref{th-main},
\begin{align*}
\psi_{[H]}(G;A) & = \vert [H] \vert \sum_{H \leq K \leq G} \mu(H,K) \vert A \vert^{[G:K]}, \\[1em]
\alpha_{[H]}(G;A)  & = \frac{\vert [H] \vert }{[G:H]}  \sum_{H \leq K \leq G} \mu(H,K) \vert A \vert^{[G:K]}.
\end{align*}
\end{corollary}

%%%%%%%%%%%%%%%%%%%%%%%%%%%%%%%%%%%%%%
\section{Configurations with normal period}

In this section we shall specialise on the case when $H$ is a normal subgroup of $G$ of finite index. In this case, the conjugacy class of $H$ just contains $H$ itself, so
\[ \psi_H(G;A) = \psi_{[H]}(G;A). \] 
Denote by $1$ the trivial subgroup. The following result has been noted in \cite[Lemma 6]{CRG19}.

\begin{lemma}\label{le-finite}
Let $G$ be any group and let $H$ be a normal subgroup of $G$ of finite index. Then,
\[\psi_H(G;A) = \psi_1(G/H;A) \ \text{ and } \ \alpha_{[H]}(G;A) = \alpha_{[1]}(G/H; A).    \]
\end{lemma}
\begin{proof}
By \cite[Proposition 1.3.7.]{CSC10}, there is a $G/H$-equivariant bijection between $A^{G/H}$ and $\Fix(H)$. Hence, configurations in $A^{G/H}$ with trivial stabiliser are in bijection with the configurations in $A^G$ with stabiliser equal to $H$.  
\end{proof}

The previous lemma allows to apply the machinery of M\"obius functions of subgroup lattices which has been developed for a variety of finite groups (e.g. see \cite{DZ20,HIO89,P93}). 

Recall that the classical M\"obius function $\tilde{\mu}$ of the poset of natural numbers $\mathbb{N}$ ordered by divisibility is given by 
\[ \tilde{\mu}(d) = \begin{cases}
0 & \text{ if $d$ has a squared prime factor} \\
1 & \text{ if $d$ is square-free with an even number of prime factors} \\
-1 & \text{ if $d$ is square-free with an odd number of prime factors}.
\end{cases}\]
Using Lemma \ref{le-finite}, the following result gives the values of $\psi_H(G;A)$ in some particular cases when $H$ is a normal subgroup of $G$. 

\begin{lemma}\label{le-formula}
Let $G$ be a finitely generated group, let $H$ be a normal subgroup of $G$ of finite index, and let $A$ be a finite set. Let $n \in \mathbb{N}$, and let $p$ and $p^\prime$ be two distinct primes.
\begin{enumerate}
\item If $G/H \cong \mathbb{Z}_n$, then $\psi_{H}(G;A) = \sum_{d \mid n} \tilde{\mu}(d) \vert A \vert^{n/d}$.
\item If $G/H \cong \mathbb{Z}_{p^k}$, then $\psi_{H}(G;A)= \vert A \vert^{p^k} - \vert A \vert^{p^{k-1}}$.
\item If $G/H \cong \mathbb{Z}_{pp^\prime}$, then $\psi_{H}(G;A)= \vert A \vert^{pp^\prime} - \vert A \vert^{p} - \vert A \vert^{p^\prime}+ \vert A \vert$. 
\item If $G/H \cong \mathbb{Z}_{p} \oplus \mathbb{Z}_p$, then $\psi_{H}(G;A) =  \vert A \vert^{p^2} - (p+1)\vert A \vert^p + p\vert A \vert$.
\end{enumerate}
\end{lemma}
\begin{proof}
Parts (1), (2) and (3) follow as it is well-known that $\mu(1, \mathbb{Z}_n) = \tilde{\mu}(n)$ (as the subgroup lattice of $\mathbb{Z}_n$ is isomorphic to the divisibility lattice of $n$). For part (4), just observe that the group $\mathbb{Z}_{p} \oplus \mathbb{Z}_p$ has $\frac{p^2-1}{p-1}=p+1$ subgroups isomorphic to $\mathbb{Z}_p$ (as each of the $p^2-1$ nontrivial elements of $\mathbb{Z}_{p} \oplus \mathbb{Z}_p$ generates a subgroup with $p-1$ nontrivial elements), which account for all its proper nontrivial subgroups. 
\end{proof}	

In the rest of this section, we shall focus on the exact determination of the small values of $\alpha_{[H]}(G;A)$. The inspiration for this question is Lemma 5 in \cite{CRG19}, which established, without using the M\"obius function, that $\alpha_{[H]}(G;A) = 1$ if and only if $[G:H]=2$ and $\vert A \vert = 2$. In general, the classification of small values of $\alpha_{[H]}(G;A)$ is relevant as it classifies configurations with small $\Aut(A^G)$-orbits, and, when $G$ is finite, it classifies the small degrees of the symmetric groups appearing in the decomposition (\ref{aut-desc}) of $\Aut(A^G)$.   

For $x \in A^G$, we have $G_x = G$ if and only if $x$ is a constant configuration. As we have precisely $\vert A \vert$ constant configurations in $A^G$, then $\alpha_{[G]}(G;A) = \vert A \vert$. Hence, we shall exclude the case $H=G$ in the following theorem. Moreover, we exclude the degenerate case $\vert A \vert = 1$.

\begin{table}[!htb]
\setlength{\tabcolsep}{5pt}
\renewcommand{\arraystretch}{2}
\centering
\begin{tabular}{|c|c|c|c|c|}
\hline \backslashbox{$G/H$}{$\vert A \vert$} & $2$ & $3$ & $4$ & $5$  \\ \hline
$\mathbb{Z}_2$ & $1$ & $3$  & $6$ & $10$ \\ \hline
$\mathbb{Z}_3$ & $2$  & $8$  & $20$ & $40$  \\ \hline
$\mathbb{Z}_2^2$ & $2$  & $15$ & $54$ & $140$ \\   \hline
$\mathbb{Z}_4$ & $3$  & $18$ & $60$ & $150$ \\   \hline
$\mathbb{Z}_5$ & $6$ & $48$ & $204$ & $624$  \\ \hline
$S_3$ & $7$ & $108$ & $650$ & $2540$  \\ \hline
$\mathbb{Z}_6$ & $9$ & $116$ & $670$ & $2580$  \\ \hline
$\mathbb{Z}_7$ & $18$ & $312$ & $2340$ & $11160$  \\ \hline
\end{tabular}
\caption{Small values for $\alpha_{[H]}(G;A)$ with $H$ normal in $G$.}\label{tab:small}
\end{table} 

\begin{theorem}
Let $G$ be a finitely generated group, let $H$ be a proper normal subgroup of $G$ of finite index, and let $A$ a finite set with at least two elements.  
\begin{enumerate}
\item $\alpha_{[H]}(G;A) = 1$ if and only if $\vert A \vert=2$ and $[G:H] = 2$. 
\item $\alpha_{[H]}(G;A) = 2$ if and only if $\vert A \vert=2$ and $[G:H] = 3$, or $\vert A \vert=2$ and $G/H \cong \mathbb{Z}_2 \oplus \mathbb{Z}_2$.
\item $\alpha_{[H]}(G;A) = 3$ if and only if $\vert A \vert=3$ and $[G:H]=2$, or $\vert A \vert=2$ and $G/H \cong \mathbb{Z}_4$. 
\item $\alpha_{[H]}(G;A) = 6$ if and only if $\vert A \vert=2$ and $[G:H] =5$, or $\vert A \vert=4$ and $[G:H]=2$.
\item $\alpha_{[H]}(G;A) = 7$ if and only if $\vert A \vert=2$ and $G/H \cong S_3$. 
\item $\alpha_{[H]}(G;A) = 8$ if and only if $\vert A \vert=3$ and $[G:H] = 3$. 
\item $\alpha_{[H]}(G;A) = 9$ if and only if $\vert A \vert=2$ and $G/H \cong \mathbb{Z}_6$. 
\item $\alpha_{[H]} (G;A)= 10$ if and only if $\vert A \vert=5$ and $[G:H]=2$. 
\item $\alpha_{[H]}(G;A) \neq 4$ and $\alpha_{[H]}(G;A) \neq 5$. 
\end{enumerate}
\end{theorem}
\begin{proof}
By Corollary 1.7.2 in \cite{GJS16}, 
\[ \vert A \vert^{[G:H]} - \vert A \vert^{[G:H]-1}	\leq \alpha_{[1]}(G/H,A) = \alpha_{[H]}(G;A).  \]
(This lower bound has been improved in Theorem 5 in \cite{CRG19}, but the above is enough for this proof). Hence, we see that $\alpha_{[1]}(G/H,A)$ is a strictly increasing function on both $[G:H]$ and $\vert A \vert$. Table \ref{tab:small} shows all values of $\alpha_{[1]}(G/H,A)$ with $[G:H] \leq 7$ and $\vert A \vert \leq 5$. Most of these values may be calculated by using the formulas of Lemma \ref{le-formula}; the only exception is the case $G/H \cong S_3$, which may be directly computed using the M\"obius function of the subgroup lattice of $S_3$ (see Figure \ref{s3-lattice}). The result follows by inspection of Table \ref{tab:small}.  

\begin{figure}
\centering
\begin{tikzpicture}[scale=1.5] 
  \node (one) at (0,2) {$S_3$};
  \node (a) at (-3,0) {$\langle(1,2)\rangle$};
  \node (b) at (-1,0) {$\langle(2,3)\rangle$};
  \node (c) at (1,0) {$\langle(1,3)\rangle$};
  \node (d) at (3,0) {$\langle(1,2,3)\rangle$};
  \node (zero) at (0,-2) {$\langle e\rangle$};
  \draw (zero) -- (a) -- (one) -- (b) -- (zero) -- (c) -- (one) -- (d) -- (zero);
\end{tikzpicture}
\caption{Subgroup lattice of $S_3$.}
\label{s3-lattice}
\end{figure}
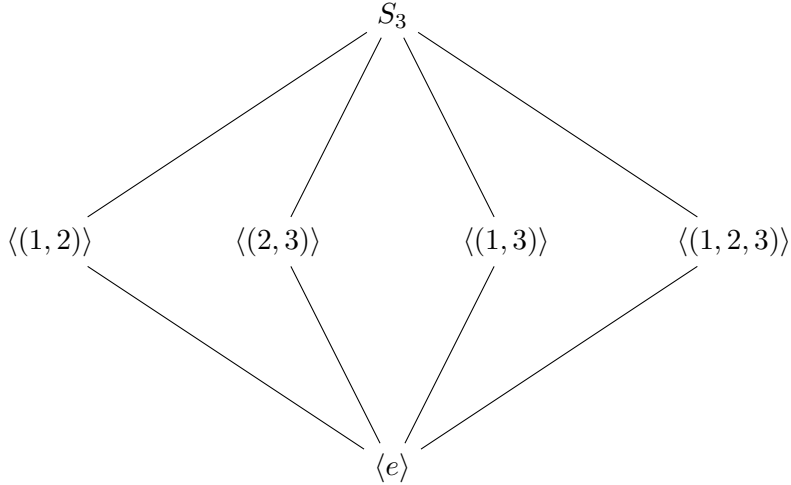

\end{proof} \bigskip

\textbf{Acknowledgments:} We sincerely thank the anonymous referee for all his precise comments that improve the quality of our manuscript. The first author of this paper was supported by a CONACYT Basic Science Grant (No. A1-S-8013) from the Government of Mexico. The second author of this paper was supported by a CONACYT National Scholarship for PhD.

\end{document}